\title{A Generalization of the Brocard Circles and the Brocard Triangles}

\documentclass[a4paper,11pt,leqno]{amsart}
\usepackage{amsfonts}
\usepackage{amssymb}
\usepackage{amsmath}
\usepackage{graphicx}
\usepackage{wrapfig}

\newtheorem{theorem}{Theorem}[section]
\newtheorem{corollary}{Corollary}[section]

\newtheorem{lemma}{Lemma}[section]
\newtheorem{proposition}{Proposition}[section]
\newtheorem{defn}{Definition}[section]

\newtheorem{rem}{Remark}[section]

\begin{document}

\begin{abstract}
In this article we provide a generalization of the Brocard circle and the Brocard triangles. The generalization arises from considering the Miquel points of two inscribed triangles having a common circumcircle. We also state some known properties which still hold in the generalized setting.
\end{abstract}
\title{}

\markboth{Sudharshan K V}{A Generalization of the Brocard Circle}

\centerline {\Large{\bf A GENERALIZATION OF THE BROCARD CIRCLE}}

\bigskip

\begin{center}
{\large SUDHARSHAN K V}

\centerline{}
\end{center}

\bigskip

\textbf{Abstract.} In this article we provide a generalization of the Brocard circle and the Brocard triangles. The generalization arises from considering the Miquel points of two inscribed triangles having a common circumcircle. We also present various properties of the Brocard triangles, which lead to a generalization of the Steiner and Tarry points.

\bigskip

\section{Introduction}
	\label{sec:intro}
The Brocard points, named after French Geometer Henri Brocard, have been a subject of significant interest and study in the field of Euclidean geometry. They have been studied extensively (see \cite{honsberger1995episodes}, \cite{guggenbuhl1953henri}). In this section, we recall the definitions of the Brocard points, the Brocard circle, and the Brocard triangles. We shall also state some known properties which are retained in the generalization developed in this paper.\\

Throughout the paper, let $ABC$ be an anticlockwise oriented triangle. 
\begin{defn}[Brocard Points]
	The first and second Brocard Points $\Omega, \Omega'$ of $\triangle ABC$ are interior points satisfying the angle conditions:$$\angle \Omega AB = \angle \Omega BC = \angle \Omega CA ~ \text{and}~ \angle \Omega'AC = \angle \Omega'BA = \angle \Omega'CB.$$
\end{defn}
We introduce circles $\omega_A, \omega_B, \omega_C$, and $\omega_A', \omega_B', \omega_C'$ as follows:

\begin{itemize}
    \item Circle $\omega_A$ passes through $A$ and is tangent to side $BC$ at $B$.
    \item Circle $\omega_B$ passes through $B$ and is tangent to side $CA$ at $C$.
    \item Circle $\omega_C$ passes through $C$ and is tangent to side $AB$ at $A$.
\end{itemize}

Similarly, in a cyclic fashion:

\begin{itemize}
    \item Circle $\omega_A'$ passes through $A$ and is tangent to side $BC$ at $C$.
    \item Circle $\omega_B'$ passes through $B$ and is tangent to side $CA$ at $A$.
    \item Circle $\omega_C'$ passes through $C$ and is tangent to side $AB$ at $B$.
\end{itemize}

\bigskip

\bigskip

\bigskip

--------------------------------------

\textbf{Keywords and phrases: }Brocard Points, Brocard Triangle, Brocard Circle

\textbf{(2020)Mathematics Subject Classification: }51P99, 60A99

\newpage

It can be observed that the first Brocard point $\Omega$ is the common point of intersection of circles $\omega_A, \omega_B, \omega_C$. On the other hand, the second Brocard point $\Omega'$ is the common point of intersection of circles $\omega_A', \omega_B', \omega_C'$. Furthermore, it is well-established that $\Omega$ and $\Omega'$ are isogonal conjugates in the reference triangle $\triangle ABC$.

\begin{defn}[Miquel Points]
	Let $D,E,F$ be points on the sidelines $BC,CA,AB$ respectively. The Miquel point of $\triangle DEF$ in $\triangle ABC$ is the common point of the circumcircles of $\triangle AEF$, $\triangle BFD$, and $\triangle CDE$. \\
 Let $A,B,C,D$ be four general points, and let $P = AB \cap CD$, $Q = AD \cap BC$. The Miquel point of quadrangle $ABCD$ is the common point of the circumcircles of $\triangle PAD$, $\triangle PBC$, $\triangle QAB$ and $\triangle QCD$.
\end{defn}

The two Brocard points of $\triangle ABC$ are therefore the Miquel points of $\triangle BCA$ and $\triangle CAB$ in $\triangle ABC$. This view of the Brocard points leads to a generalization (Section \ref{subsec:prelim}), where we study Miquel points of two inscribed triangles that share a circumcircle.

\begin{defn}[Brocard Triangles]
	(See \cite{weisstein2001brocard}, \cite{weisstein2003first}, and \cite{weisstein2003second}) The vertices of the first Brocard triangle are defined as $A_1 = \Omega B \cap \Omega' C$, $B_1 = \Omega C \cap \Omega' A$, and $C_1=\Omega A \cap \Omega' B$.\\
    On the other hand, the second Brocard triangle has its vertices denoted as $A_2=\omega_A \cap \omega_A', B_2=\omega_B \cap \omega_B'$, and $C_2=\omega_C\cap \omega_C'$, where these intersection points are distinct from $A$, $B$, and $C$.
\end{defn}

\begin{figure}[!h]
	\includegraphics[scale=0.35]{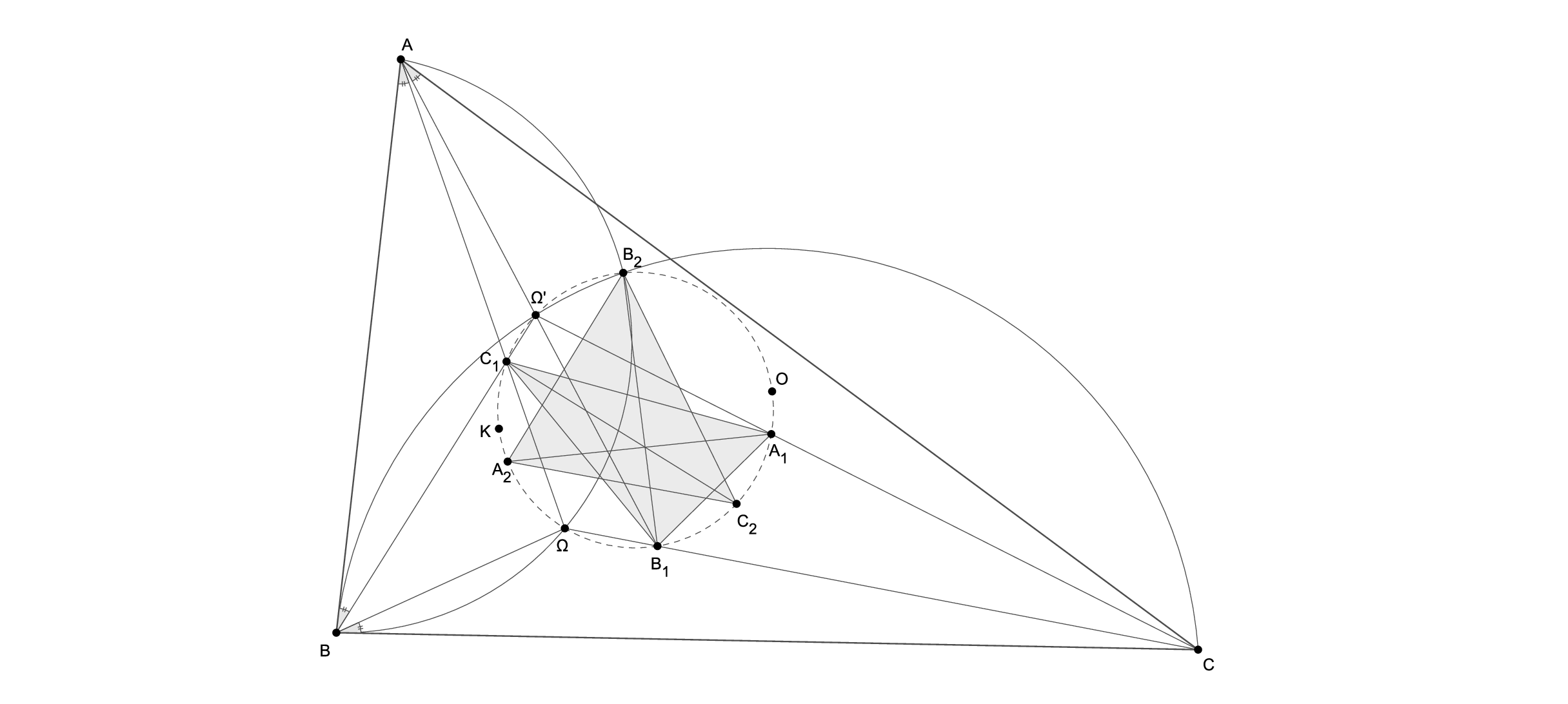}
	\caption{Brocard Points and associated Objects}
\end{figure}

It is known that $\triangle A_1B_1C_1$ and $\triangle A_2B_2C_2$ have the same circumcircle. This circle is called the Brocard Circle, and we know that the circumcenter $O$ and symmedian point $K$ of $\triangle ABC$ are diametrically opposite points in this circle. \\
The construction we present to generalize the Brocard points extends to a generalization of the Brocard circle and the Brocard triangles, as detailed in Section \ref{subsec:bro_cir}.

\section{A Generalization of the Brocard Points}
Firstly, we present several well-known and useful Lemmas.
\subsection{Preliminary Lemmas}
\begin{lemma}
	\label{lemma:spiral}
	Let $M$ be the Miquel point of $\triangle DEF$ in $\triangle ABC$. Then, $\triangle DEF$ is the image of the pedal triangle of $M$ in $\triangle ABC$ under a spiral similarity centered at $M$. 
\end{lemma}
\begin{proof}
    Throughout the paper, we represent the directed measure (modulo $\pi$) of an angle $\theta$ as $\measuredangle \theta$.\\

    Let $\triangle D'E'F'$ be the pedal triangle of point $M$ with respect to $\triangle ABC$. Since $M$ is the Miquel point, we observe that $\measuredangle MDB = \measuredangle MFA = \measuredangle MEC$. Additionally, $\measuredangle MD'B = \measuredangle MF'A = \measuredangle ME'C = 90^\circ$. Rewriting the angle relations, we find that $\measuredangle MDD' = \measuredangle MEE' = \measuredangle MFF'$ and $\measuredangle MD'D = \measuredangle ME'E = \measuredangle MF'F = 90^\circ$.\\
 
    As a consequence of these angle relationships, we establish that $\triangle MDD' \stackrel{+}{\sim} \triangle MEE' \stackrel{+}{\sim} \triangle MFF'$. This similarity leads to the existence of a spiral similarity $\Psi_M$ centered at point $M$, which maps $\triangle D'E'F'$ to $\triangle DEF$.
\end{proof}
\begin{lemma}
	\label{lemma:cyclic}
	Let $M$ be the Miquel point of $\square ABCD$, where $A,B,C,D$ lie on a circle $\odot(O,r)$. Let $P=AB\cap CD$, $Q=AD\cap BC$. Then $OM\perp PQ$ and $M\in PQ$.
\end{lemma}
\begin{figure}[!h]
	\includegraphics[scale=0.35]{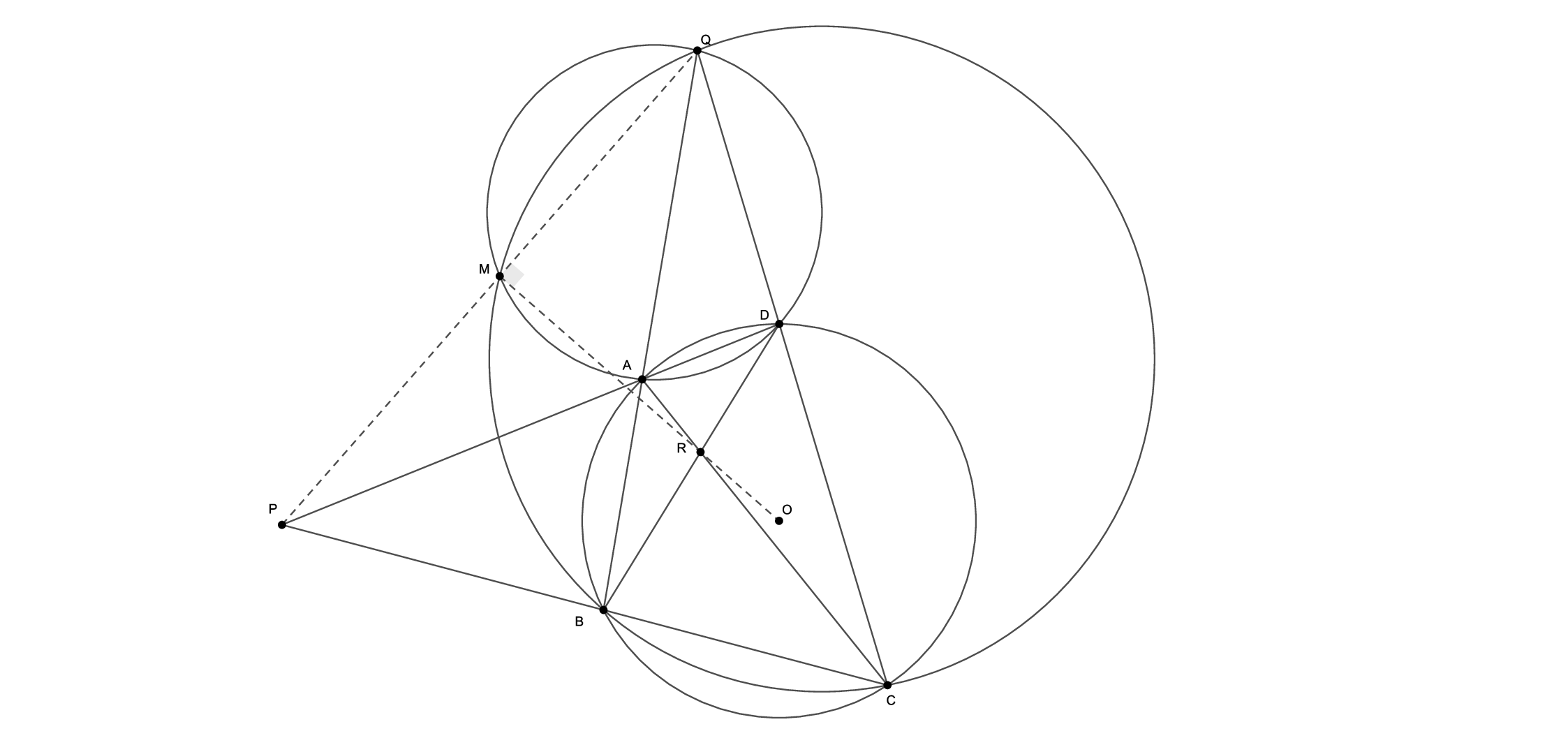}\\   
	\caption{Miquel Point of a Cyclic Quadrangle}
\end{figure}
\begin{proof}	
    Consider the points $R = AC \cap BD$ and $R^*$, which is the inverse of $R$ in the circumcircle of quadrilateral $ABCD$ which we denote by $\odot(ABCD)$. We shall show that $R^* = M$.\\
    
    Let $\overline{\chi}$ represent the directed length of a segment $\chi$. It is evident that $\overline{RR^*} \cdot \overline{RO} = RO^2 - r^2 = \overline{RA} \cdot \overline{RC}$, which implies that $R^*, O, A, C$ are concyclic. Similarly, $R^*, O, B, D$ are concyclic. Through some angle-chasing, we find that $\measuredangle AR^*B = \measuredangle AR^*O + \measuredangle OR^*B = \measuredangle ACO + \measuredangle ODB = 90^\circ - \measuredangle CDA + 90^\circ - \measuredangle BCD = \measuredangle QCD + \measuredangle CDQ = \measuredangle AQB$, which leads to $\measuredangle AR^*B = \measuredangle AQB$. It follows that $\square AR^*QB$ is cyclic. Similarly, $\square AR^*PD$, $\square BR^*CP$, and $\square CR^*DQ$ are cyclic, which gives us $R^* = M$.\\
    
    Finally, applying Brokard's theorem on $\square ABCD$ shows that $OM$ is an altitude of $\triangle POQ$. 
\end{proof}

\begin{lemma}
	\label{lemma:Simson}
	The angle between Simson lines $\ell_M, \ell_N$ of two points $M,N$ in $\odot (ABC)$ is half the measure of arc ${MN}$. More precisely, $\measuredangle (\ell_M, \ell_N) = \measuredangle MAN$.
\end{lemma}
The proof of the final lemma is omitted. It can be found in any standard text, such as \cite{chen2021euclidean}.
\subsection{Generalized Brocard Points}
\label{subsec:prelim}
Let a circle $\Gamma$ centered at $O$ intersect sides $BC,CA$, and $AB$ of a $\triangle ABC$ (oriented anticlockwise) at $\{A_1,A_2\}$, $\{B_1, B_2\}$, and $\{C_1, C_2\}$. Let $P$ and $Q$ be the Miquel Points of $\triangle A_1B_1C_1$, $\triangle A_2B_2C_2$ in $\triangle ABC$ respectively. 
\begin{theorem}[see \cite{alastor2020moody}]
	\label{thm:iso_conj}
    The points $P$ and $Q$ are isogonal conjugates in $\triangle ABC$, and we have an equality of angles $\measuredangle PA_1A_2 = \measuredangle PB_1B_2$ $=\measuredangle PC_1C_2$ $=-\measuredangle QA_2A_1$ $=-\measuredangle QB_2B_1$ $=\measuredangle QC_2C_1$. 
\end{theorem} 
\begin{figure}[!h]
	\includegraphics[scale=0.3]{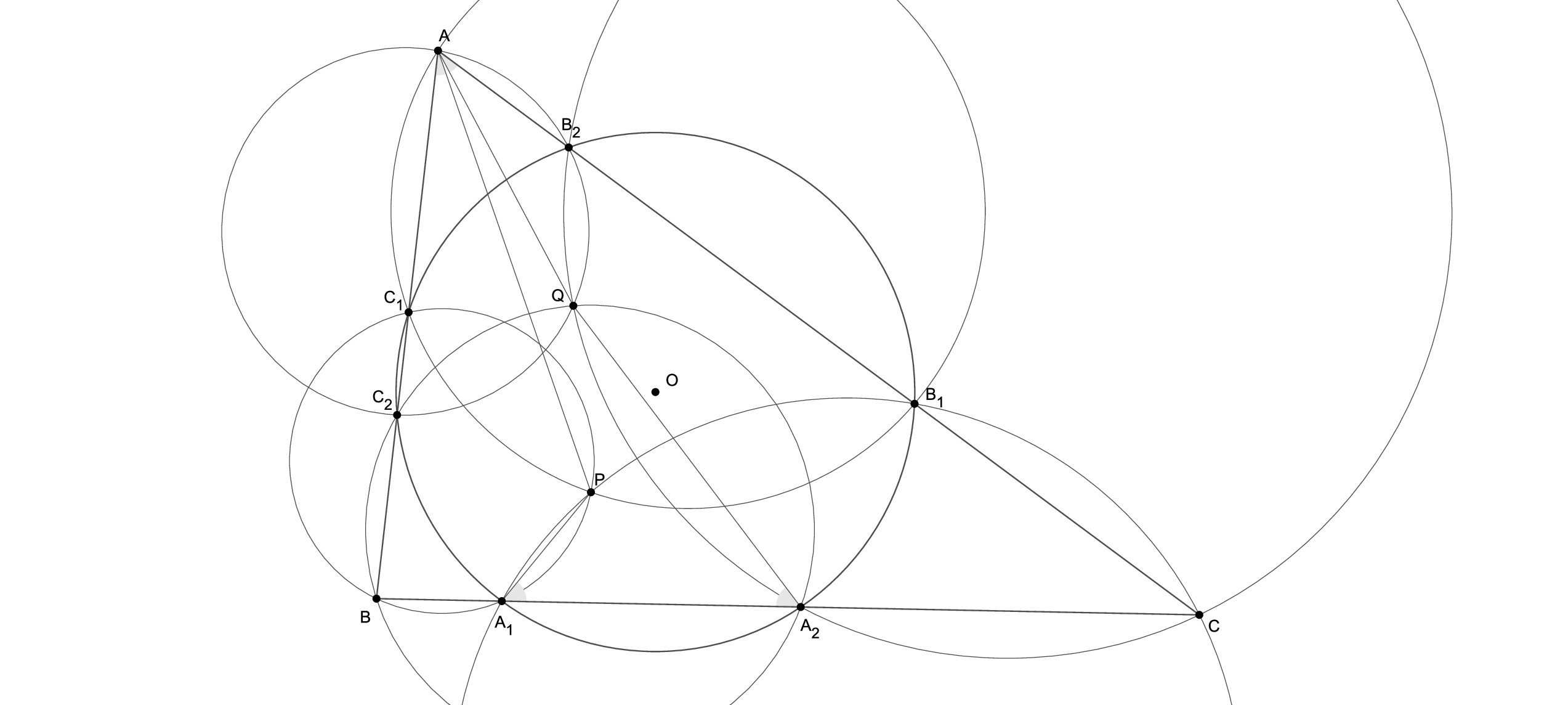}
	\caption{Generalized Brocard Points}
\end{figure}
\begin{proof}
	We provide a combined proof for both parts, inspired by Minjae Kwon's proof in \cite{kwon}. Let us redefine points $Q$, $A_2$, $B_2$, and $C_2$ from the given point $P$.

    First, we redefine $Q$ as the isogonal conjugate of $P$ in $\triangle ABC$. Next, we redefine points $A_2$, $B_2$, and $C_2$ as new points on $BC$, $CA$, and $AB$, respectively, such that the following conditions hold:
        \begin{align*}
        \measuredangle PA_1A_2 &= -\measuredangle QA_2A_1, \\
        \measuredangle PB_1B_2 &= -\measuredangle QB_2B_1, \\
        \measuredangle PC_1C_2 &= -\measuredangle QC_2C_1.
        \end{align*}

    Since $P$ is the Miquel point of $\triangle A_1B_1C_1$, it follows that $-\measuredangle PA_1C = -\measuredangle PB_1A = -\measuredangle PC_1B = \measuredangle QA_2C = \measuredangle QB_2A = \measuredangle QC_2B$. This equality allows us to deduce that $Q$ is the Miquel Point of $\triangle A_2B_2C_2$.\\

    We now proceed to show that $A_1$, $A_2$, $B_1$, $B_2$, $C_1$, and $C_2$ are concyclic. From the new definitions, we have:
        \begin{align*}
        \measuredangle CA_1P &= -\measuredangle CB_2Q, \\
        \measuredangle PCA_1 &= -\measuredangle QCB_2.
        \end{align*}

    Thus, we can observe that $\triangle CPA_1 \stackrel{-}{\sim} \triangle CQB_2$. Analogously, we have $\triangle CQA_2 \stackrel{-}{\sim} \triangle CPB_1$, leading to the relation $$\frac{CA_1}{CB_2} = \frac{CB_1}{CA_2} = \frac{CP}{CQ}.$$ Consequently, $CA_1\cdot CA_2 = CB_1 \cdot CB_2,$ which implies that the quadrangle $\square A_1A_2B_1B_2$ is cyclic. Similarly, one shows that the quadrangles $\square B_1B_2C_1C_2$ and $\square A_1A_2C_1C_2$ are also cyclic.\\
    
    If the circles $\odot(A_1A_2B_1B_2)$, $\odot(A_1A_2C_1C_2)$, and $\odot(B_1B_2C_1C_2)$ were distinct, their pairwise radical axes would meet at a point, which would lead to a contradiction. Hence, the three circles coincide, confirming that $A_1$, $A_2$, $B_1$, $B_2$, $C_1$, and $C_2$ lie on a single circle. These new definitions of $A_2$, $B_2$, and $C_2$ agree with the original definitions, thereby completing the proof.
\end{proof}

\begin{rem} Consider the spiral similarities $\Psi_P, \Psi_Q$ centered at $P,Q$ taking their pedal triangles to $\triangle A_1B_1C_1$ and $\triangle A_2B_2C_2$ respectively. Using Theorem~\ref{thm:iso_conj}, we observe that the (directed) angles of rotation $\theta_P, \theta_Q$ associated with $\Psi_P, \Psi_Q$ bear the relation $$\theta_P+\theta_Q=0.$$
\end{rem}

\section{A Generalization of the Brocard Circle and Triangles}
We now build on the configuration from Section \ref{subsec:prelim}, by defining the three points $$T_A = PA_1\cap QA_2, T_B= PB_1\cap QB_2, \text{ and } T_C= PC_1\cap QC_2.$$ We also introduce three other points $$A'=\odot (AB_1C_1) \cap \odot (AB_2C_2),$$ $$B'=\odot (BC_1A_1) \cap \odot (BC_2A_2),$$ \text{and} $$C'=\odot (CA_1B_1) \cap (CA_2B_2).$$ 

\begin{proposition}
	\label{thm:rad_ax}
	The lines $AA', BB', CC'$ are concurrent.
\end{proposition}
\begin{proof}	
    The points $A_0 = B_1C_2 \cap C_1B_2$, $B_0 = A_1C_2 \cap C_1A_2$, and $C_0 = A_1B_2 \cap B_1A_2$ lie on the same line, due to Pascal's theorem applied to the hexagon $A_1B_1C_1A_2B_2C_2$.\\    
    
    Now, let $R$ be the pole of the line $A_0B_0C_0$ with respect to the circle $\Gamma$. By applying the Radical Axis theorem to $\odot(AB_1C_1)$, $\odot(AB_2C_2)$, and $\Gamma$, we find that lines $B_1C_1$, $B_2C_2$, and $AA'$ are concurrent. Moreover, from Brokard's theorem applied to the quadrilateral $B_1C_1B_2C_2$, we deduce that $AA'$ is the polar of point $A_0$ with respect to $\Gamma$. As $AA'$ passes through $R$ due to La Hire's theorem, similarly, $BB'$ and $CC'$ pass through $R$. Thus, we conclude that the lines $AA'$, $BB'$, and $CC'$ are concurrent at point $R$.
\end{proof}
\subsection{Generalization of the Brocard Circle}
We present the generalization of the Brocard circle in the following theorem.
\label{subsec:bro_cir}
\begin{theorem}
	\label{thm:gen_brocard}
	The points $P,Q,O,R,A',B',C',T_A,T_B,T_C$ lie on a circle. 
\end{theorem}
\begin{figure}[!h]
	\includegraphics[scale=0.31]{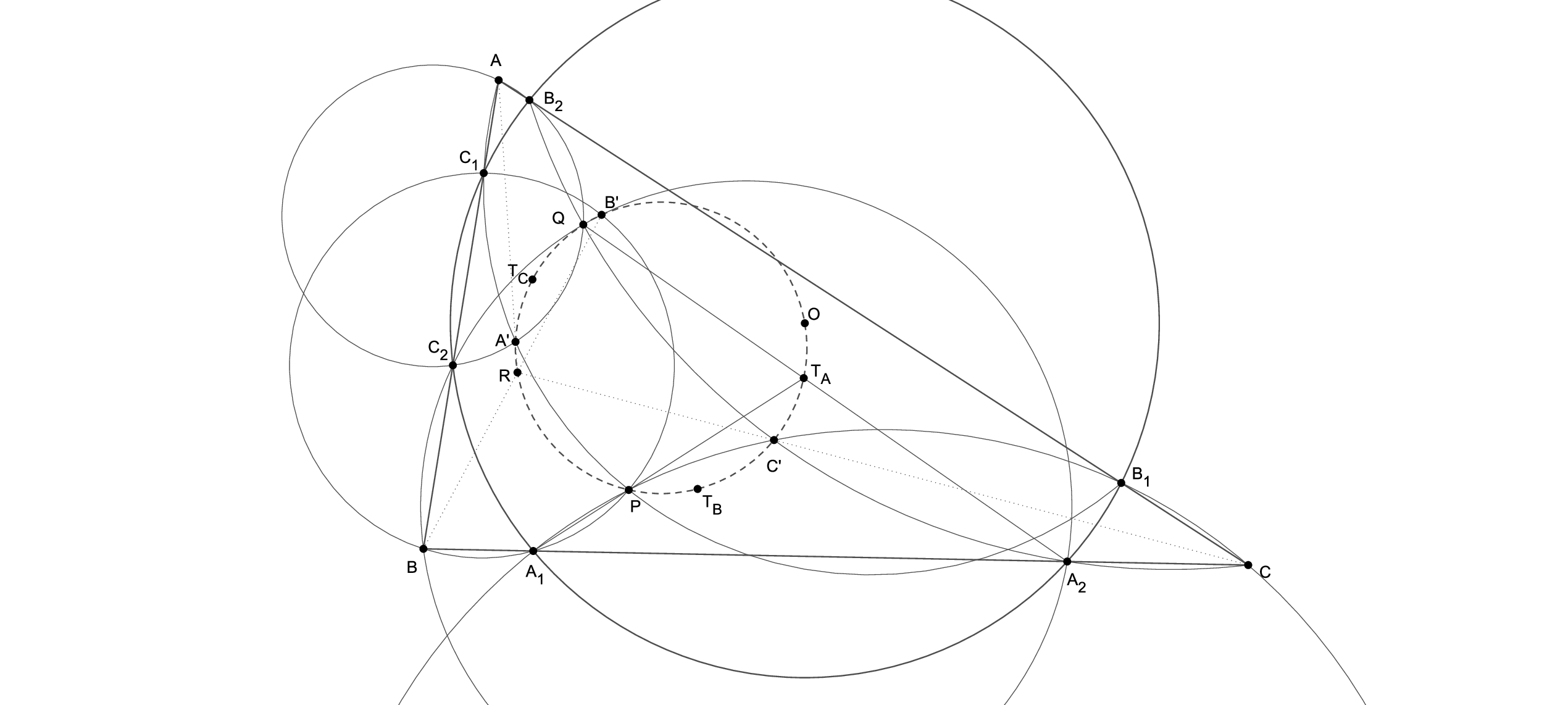}\\
	\caption{Generalized Brocard Circle}
\end{figure}
\begin{proof}
	Observe that $A'$ is the Miquel point of $\square C_1C_2B_2B_1$ since $A'=\odot(AB_1C_1)\cap \odot(AB_2C_2)$. By applying Lemma~\ref{lemma:cyclic}, we deduce that $\angle OA'A=90^\circ$. Considering that $AA'$ passes through $R$, we can infer that $\angle OA'R=90^\circ$. Similarly, $\angle OB'R=\angle OC'R=90^\circ$, from which we can conclude that $A',B',C'$ lie on the circle with diameter $OR$.\\

    Upon chasing angles, we find $\measuredangle PA'Q=\measuredangle PA'C-\measuredangle QA'C=\measuredangle PA_1C - \measuredangle QA_2C$. Employing Theorem~\ref{thm:iso_conj}, we have $\measuredangle PA_1C=-\measuredangle QA_2C$, which yields $\measuredangle PA'Q=\theta_Q-\theta_P$ (also equal to $2\theta_Q$). As this expression is symmetric in $A,B,C$, it follows that $\measuredangle PA'Q=\measuredangle PB'Q=\measuredangle PC'Q$. Thus, $P$ and $Q$ lie on the circle $\odot(A'B'C')$.\\

    Furthermore, we observe that $\measuredangle PT_AQ = \measuredangle PA_1A_2 + \measuredangle A_1A_2Q = \theta_Q-\theta_P = \measuredangle PA'Q$. Consequently, $T_A$ lies on $\odot(A'B'C')$. Similarly, we can deduce that $T_B$ and $T_C$ also lie on $\odot(A'B'C')$. Consequently, the points $P,Q,O,R,A',B',C',T_A,T_B,T_C$ are all concyclic.
\end{proof}

\begin{corollary}
	\label{corr:equidistant}
	The points $P$ and $Q$ are equidistant from $R$.
\end{corollary}
\begin{proof}
    We can observe from Theorem~\ref{thm:iso_conj} that $\triangle T_AA_1A_2$ is isosceles, with $T_AA_1=T_AA_2$. Since $OT_A$ is the perpendicular bisector of $A_1A_2$, $OT_A$ is the angle bisector of $\angle A_1T_AA_2$. So, $OT_A$ bisects $\angle PT_AQ$. Also, $O$ lies on the circumcircle of $\triangle T_APQ$. Therefore, $OP=OQ$. Since $OR$ is a diameter of this circle, the result follows.
\end{proof}

\begin{rem} Minjae Kwon (in \cite{kwon}) shows a more general result: For $\triangle DEF$ and $\triangle XYZ$ inscribed in $\triangle ABC$, where $\{D,X\}, \{E,Y\}$, and $\{F, Z\}$ are points on $BC,CA$, and $AB$ respectively.) Suppose that the perpendicular bisectors of $DX,EY,CZ$ concur at a point, $T$. If $O_1, O_2$ are the Miquel points of $\triangle DEF, \triangle XYZ$, then $TO_1=TO_2$.
\end{rem}
\subsection{Revisiting the Brocard Points, Circle and Triangles}
	\label{subsec:revisit}
As discussed in Section \ref{sec:intro}, a special case arises when $A_1=B$, $B_1=C$, $C_1=A$, $A_2=C$, $B_2=A$, $C_2=B$. In this scenario, the points $P$ and $Q$ coincide with the first and second Brocard points, $\Omega$ and $\Omega'$, respectively, of $\triangle ABC$. The circle $\Gamma$ transforms into the circumcircle of $\triangle ABC$, and $O$ becomes the circumcenter. Additionally, the points $A_0$, $B_0$, and $C_0$ correspond to the intersections of the tangents to the circumcircle at the vertices of $\triangle ABC$ with the opposite sides. Then, the symmedian point $K$ of $\triangle ABC$ coincides with the pole of the line $A_0B_0C_0$ with respect to the circumcircle, thus making $R=K$.\\

Furthermore, we observe that the triangles $\triangle T_AT_BT_C$ and $\triangle A'B'C'$ become the first and second Brocard triangles, respectively. One can now see the proofs of several results from Section \ref{sec:intro}. Indeed, applying Theorem~\ref{thm:gen_brocard} demonstrates that the circle with diameter $OK$ passes through $\Omega$ and $\Omega'$ and also through the vertices of the first and second Brocard Triangles. Moreover, Corollary~\ref{corr:equidistant} reveals that $OK$ is the perpendicular bisector of $\Omega\Omega'$.

\section{Properties of the Generalized Brocard Triangles}

Various properties of the Brocard triangles carry over to the generalization. We present a few in this section.
\begin{proposition}
	\label{prop:inv_sim}
	We have a similarity $\triangle T_AT_BT_C \stackrel{-}{\sim} \triangle ABC$ of the (generalized) first Brocard triangle with the reference triangle.
\end{proposition}
\begin{proof}
	Observe that $OT_A\perp BC$ since $OT_A$ is the perpendicular bisector of $A_1A_2$. Similarly, $OT_B\perp AC$, indicating that $\measuredangle T_AOT_B = \measuredangle BCA$. By utilizing Theorem~\ref{thm:gen_brocard}, we find that $\measuredangle T_AOT_B = \measuredangle T_AT_CT_B$, resulting in $\measuredangle T_AT_CT_B = -\measuredangle ACB$. Analogously, we obtain similar relationships, leading to the similarity $\triangle T_AT_BT_C \stackrel{-}{\sim} \triangle ABC$.
\end{proof}
\begin{proposition}[Generalized Steiner Point]
	The parallels from $A,B,C$ to $T_BT_C$, $T_CT_A$, $T_AT_B$ are concurrent on $\odot (ABC)$. 
\end{proposition}
\begin{proof}
	Consider the parallels drawn from points $A$ and $B$ to the lines $T_BT_C$ and $T_CT_A$, respectively, which intersect at a point $S_t$. We can observe that $\measuredangle AS_tB = \measuredangle T_BT_CT_A = \measuredangle ACB$, where the last equality follows from Proposition~\ref{prop:inv_sim}. Consequently, $S_t$ lies on the circumcircle of $\triangle ABC$.\\

Further, we have $\measuredangle AS_tC = \measuredangle ABC = \measuredangle T_CT_BT_A$. As $AS_t \parallel T_BT_C$, it follows that $CS_t \parallel T_AT_B$. This leads us to the conclusion that the parallels drawn from points $A$, $B$, and $C$ to the lines $T_BT_C$, $T_CT_A$, and $T_AT_B$, respectively, are concurrent at point $S_t$, which lies on $\odot (ABC)$.
\end{proof}
\begin{corollary}[Generalized Tarry Point]
	The perpendiculars from $A,B,C$ to $T_BT_C$, $T_CT_A$, $T_AT_B$ are concurrent on the antipode $T_a$ of $S_t$ in $\odot (ABC)$.
\end{corollary}
\begin{proposition}
	\label{prop:inv_sim_2}
	We have a similarity of polygons $$\triangle ABC \cup S_t\cup T_a \stackrel{-}{\sim} \triangle T_AT_BT_C \cup R \cup O.$$
\end{proposition}

\begin{proof}
	In the proof of Theorem~\ref{thm:gen_brocard}, we established that $OR$ is the diameter of $\odot(T_AT_BT_C)$ and $OT_A\perp BC$, implying $RT_A\parallel BC$. Additionally, as $AS_t\parallel T_BT_C$, we find that $\measuredangle (\overline{BC},\overline{AS_t}) = - \measuredangle (\overline{T_BT_C},\overline{T_AR})$. Then, utilizing the similarity in Proposition~\ref{prop:inv_sim}, we arrive at the similarity $$\triangle ABC \cup S_t \stackrel{-}{\sim} \triangle T_AT_BT_C \cup R.$$\\ 
    Further, the pairs ${T_a,S_t}$ and ${O,R}$ are antipodal points in $\odot(ABC)$ and $\odot(T_AT_BT_C)$, respectively. Hence, one obtains the similarity $$\triangle ABC \cup S_t \cup T_a \stackrel{-}{\sim} \triangle T_AT_BT_C \cup R \cup O.$$
\end{proof}
\begin{figure}
	\includegraphics[scale=0.3]{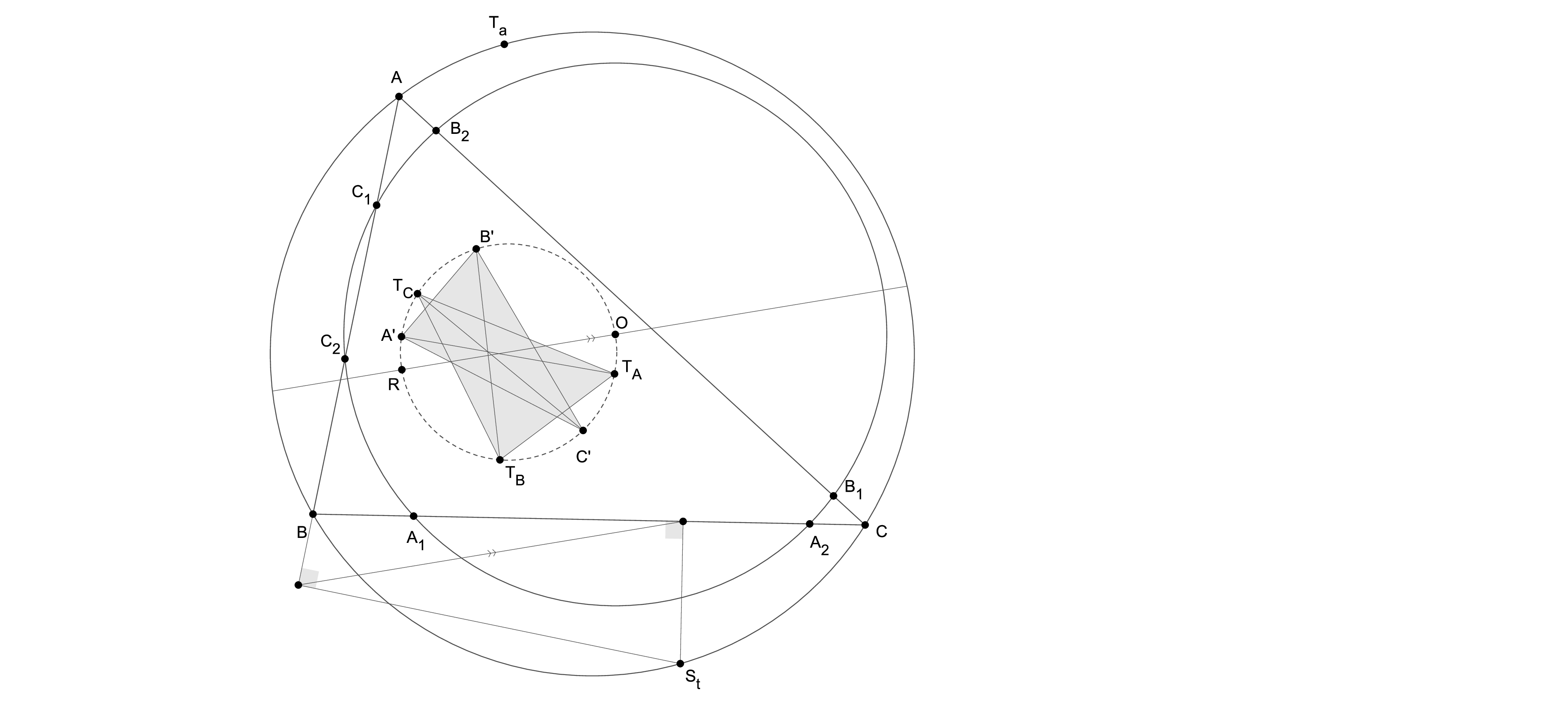}\\
	\caption{Generalized Brocard Triangles and their properties}
\end{figure}

\begin{proposition}
	The triangles $ T_AT_BT_C$ and $ A'B'C'$ are perspective. 
\end{proposition}
\begin{proof}
	Let $R^*$ denote the isogonal conjugate of $R$ in $\triangle ABC$. Introduce a point $S$ such that $\triangle T_AT_BT_C \cup S \stackrel{-}{\sim} \triangle ABC \cup R^*$. From the proof of Theorem~\ref{thm:gen_brocard}, we know that $OB'\perp BR$ and $OT_A\perp BC$. Consequently, we obtain $\measuredangle T_AOB' = -\measuredangle RBC$. As $\measuredangle T_AT_BB' = \measuredangle T_AOB'$ and $\measuredangle R^*BA = -\measuredangle RBC$, it follows that $\measuredangle R^*BA = \measuredangle T_AT_BB'$. \\
    
    By definition of point $S$, we have $-\measuredangle ABR^* = \measuredangle T_AT_BS$. Thus, we find that $\measuredangle T_AT_BS=\measuredangle T_AT_BB'$, implying that $T_BB'$ passes through $S$. Similarly, we can establish that $T_AA'$ and $T_CC'$ also pass through $S$. Consequently, we can assert that $\triangle T_AT_BT_C$ and $\triangle A'B'C'$ are perspective.   
\end{proof}

\begin{proposition}
 \label{prop:simson2}
	The Simson line of $S_t$ is parallel line $OR$.
\end{proposition}
\begin{proof}
	Consider the Simson line $\ell_{S_t}$ of point $S_t$ with respect to $\triangle ABC$. Let $\ell_A$ represent the $A$-altitude of $\triangle ABC$, which is also the Simson line of point $A$ in $\triangle ABC$. Utilizing Lemma~\ref{lemma:Simson}, we find that $$\measuredangle(\ell_{S_t},\overline{BC}) = \measuredangle (\ell_{S_t},\ell_A) - 90^\circ = \measuredangle AS_tT_a = \measuredangle ORT_A,$$ where the last equality follows from Proposition~\ref{prop:inv_sim_2}. As proven in Proposition~\ref{prop:inv_sim_2}, we have $RT_A\parallel BC$. Consequently, we deduce that $OR \parallel \ell_{S_t}$, showcasing the parallel relationship between the line $OR$ and the Simson line $\ell_{S_t}$.
\end{proof}
\begin{corollary}
	The Simson line of $T_a$ is perpendicular to $OR$.
\end{corollary}
\begin{proposition}
    \label{prop:simson}
    Let $OR$ intersect $BC, CA$, and $AB$ at $X,Y,Z$. Let $O_A,O_B,$ and $O_C$ be the circumcenters of $AYZ, BZX,$ and $CXY$. Triangles $ABC$ and $O_AO_BO_C$ are perspective at $S_t$.
\end{proposition}
\begin{proof}
    By using Proposition~\ref{prop:simson2}, one derives an angle equality:
   \begin{align*}
        \measuredangle YZA &= \measuredangle (\ell_{S_t}, AB) \\
        &= 90^\circ - \measuredangle S_tBC \\
        &= 90^\circ - \measuredangle S_tAC.
    \end{align*}
    We can now conclude that $AS_t$ passes through the circumcenter $O_A$ of $\triangle AYZ$, after which it is clear that $ABC$ and $O_AO_BO_C$ are perspective at point $S_t$.
\end{proof}

\bibliographystyle{plain}
\bibliography{references}

\begin{thebibliography}{1}

\bibitem{chen2021euclidean}
Evan Chen.
\newblock {\em Euclidean geometry in mathematical olympiads}, volume~27.
\newblock American Mathematical Soc., 2021.

\bibitem{guggenbuhl1953henri}
Laura Guggenbuhl.
\newblock Henri brocard and the geometry of the triangle.
\newblock {\em The Mathematical Gazette}, 37(322):241--243, 1953.

\bibitem{honsberger1995episodes}
Ross Honsberger.
\newblock {\em Episodes in nineteenth and twentieth century Euclidean
  geometry}, volume~37.
\newblock Cambridge University Press, 1995.

\bibitem{alastor2020moody}
Mmukul Khedekar.
\newblock Beautiful geometry dealing with miquel points , isogonal conjugates.
\newblock {\em
  {\texttt{https://artofproblemsolving.com/community/c6h2060398}}}, 2020.

\bibitem{kwon}
Minjae Kwon.
\newblock {\em {\texttt{lminsl.github.io/archives}}}, 2020.

\bibitem{weisstein2001brocard}
Eric~W Weisstein.
\newblock Brocard points.
\newblock {\em https://mathworld. wolfram. com/}, 2001.

\bibitem{weisstein2003first}
Eric~W Weisstein.
\newblock First brocard triangle.
\newblock {\em \texttt{https://mathworld. wolfram. com/}}, 2003.

\bibitem{weisstein2003second}
Eric~W Weisstein.
\newblock Second brocard triangle.
\newblock {\em \texttt{https://mathworld. wolfram. com/}}, 2003.

\end{thebibliography}

\bigskip

\bigskip

\bigskip

DEPARTMENT OF MATHEMATICS

INDIAN INSTITUTE OF SCIENCE

BANGALORE, INDIA

\textit{E-mail address}: \texttt{sudharshankv02@gmail.com}







\newpage

\end{document}